\documentclass[reqno,12pt]{amsart}

\usepackage[margin=1.25in]{geometry}

	\usepackage{comment}

	\usepackage[utf8]{inputenc}

	\usepackage{amsmath,amsfonts,amssymb,amsthm}

	\usepackage[T1]{fontenc}

	\usepackage{etoolbox}
	\patchcmd{\section}{\scshape}{\scshape\bfseries}{}{}
	\makeatletter
	\renewcommand{\@secnumfont}{\scshape\bfseries}
	\makeatother

	\let\epsilon\varepsilon

	\usepackage{parskip}

	\usepackage[bookmarks=true,bookmarksopen=true]{hyperref}

    \usepackage{appendix}

    \usepackage{multicol}
    
    \usepackage{blkarray} 

    \usepackage{subcaption}

	\newtheorem{theorem}{Theorem}    
	
	\newtheorem{conjecture}[theorem]{Conjecture}
	
	\newtheorem{corollary}[theorem]{Corollary}
        
	\newtheorem{lemma}[theorem]{Lemma}
	
	\newtheorem{main}{Theorem}

	\numberwithin{theorem}{section}
	
	\theoremstyle{plain}

	\theoremstyle{definition}
	
	\newtheorem{example}[theorem]{Example}
	
	\theoremstyle{remark}

	\numberwithin{equation}{section}
	
	\numberwithin{table}{section}

	\usepackage{tikz}
	\usetikzlibrary{matrix}
    \usetikzlibrary{calc}

    \usepackage[all,cmtip]{xy}

	\usepackage{todonotes}

	\usepackage{enumerate}

	\usepackage{float}

	\makeatletter
	\def\blfootnote{\gdef\@thefnmark{}\@footnotetext}
	\makeatother

	\DeclareMathOperator{\codim}{codim}		\DeclareMathOperator{\cod}{\codim}

	\DeclareMathOperator{\Ric}{Ric}

	\newcommand{\of}[1]{\left( #1 \right)}

	\newcommand{\gH}{\mathsf{H}}

	\newcommand{\SO}{\mathsf{SO}}

	\newcommand{\gT}{\mathsf{T}}

	\newcommand{\bZ}{\mathbb{Z}}	\newcommand{\Z}{\bZ}

\date{\today}

\title{Weighted coloop splittings in rank six}

\author{James Dylan Douthitt}
\address{Department of Mathematics, Syracuse University}
\email{email: jddouthi@syr.edu}

\author{Lee Kennard}
\address{Department of Mathematics, Syracuse University}
\email{email: ltkennar@syr.edu}

\author{Josef Komissar}
\address{Department of Mathematics, Syracuse University}
\email{email: jskomiss@syr.edu}

\begin{document}

\begin{abstract}
With a view toward applications in Riemannian geometry, we explore coloop splitting properties of regular matroids. Nienhaus showed by classification in rank four that a regular matroid has a cocircuit whose deletion yields two coloops unless the matroid takes a particular form. In the latter case, one can split off any element of the ground set as a coloop. We reprove this using Seymour's structure theorem for regular matroids and prove an extension to matroids of ranks five and six. As an application to Riemannian geometry, we prove that the torus symmetry assumption in a recent result of Mouill\'e, Nienhaus, and the second author can be relaxed from rank ten to rank nine.
\end{abstract}
\maketitle

In \cite{QiaoZhang10}, the authors prove every simple graph with at least four vertices and minimum degree three has a cycle with at least two chords. Equivalently, there exists a cycle whose contraction produces two loops. 
As we observe in this paper, this result holds without the assumption of simple if we add the assumption of 2-vertex-connectivity, an assumption which is natural for our applications to matroid results.

\begin{main}\label{thm:BonusColoop-Cographic}
    If $G$ is a $2$-vertex-connected graph with minimum degree at least three, then there exists a cycle $C$ such that $G/C$ has a vertex with two loops. %has two loops.$G$ has a cycle with at least two chords.
\end{main}

For a cographic matroid, this result translates to the following statement: {\it If $M$ is a simple connected cographic matroid, then there is a cocircuit whose deletion leaves two coloops.} The main technical result of this paper is to prove a graphic version of this statement when the graph is not the complete graph. The statement involves replacing the assumption on the minimum degree by one on the girth. Following the notation of \cite{Diestel-book}, we write $K_d$ to denote the complete graph on $d$ vertices and $P_n$ to denote the path on $n$ edges.

\begin{main}\label{thm:BonusColoop-Graphic}
    Let $G$ be a $2$-vertex-connected simple graph. 
    \begin{enumerate}
        \item
        \label{thm:BonusColoop-Graphic_notKd} 
        If $G\neq K_d$, then there exists a non-empty minimal cut set $K$ of edges such that $G\backslash K = H\oplus P_2$ for a subgraph $H$ of $G$.
        \item
        \label{thm:BonusColoop-Graphic_Kd}
        If $G=K_d$ and $e \in E(G)$, then there exists a minimal cut set $K$ of edges such that $G\backslash K = H\oplus P_1$ for a subgraph $H$ of $G$, where $E(P_1) = \{e\}$.
    \end{enumerate}
\end{main}

Using the classification of regular matroids in \cite{Seymour80} and the explicit descriptions in small ranks in \cite{DanilovGrishukhin99, Israel-PhD}, we prove the following as a consequence of these results:

\begin{main}\label{thm:BonusColoop}
    If $M$ is a simple regular matroid with $3 \leq r(M)\leq 6$, then one of the following holds.
    \begin{enumerate}
        \item[(i)] There exists a cocircuit $C^*$ such that $M\backslash C^* = U_{2,2}\oplus N$, where $N$ is a minor of $M$ of rank $r(M)-3$.
        \item[(ii)]  For all $e\in E(M)$, there exists a cocircuit $C^*$ such that $M\backslash C^* = U_{1,1}\oplus N$, where $E(U_{1,1}) = \{e\}$ and $N$ is a minor of $M$ of rank $r(M)-2$.
    \end{enumerate}
\end{main}

In rank four, this follows from results of Nienhaus, who used an analysis of the $17$ simple regular matroids of rank four (see \cite{Nienhaus-pre}). As we explain below, our proof method is via Seymour's theorem on regular matroids and, in particular, provides a second proof of Nienhaus' results. As we will see, Condition (i) of Theorem \ref{thm:BonusColoop} holds unless $M$ is the cycle matroid of a cluster graph (see Theorem \ref{thm:BonusColoop+}).

Theorem~\ref{thm:BonusColoop} and Nienhaus’ results are both refinements in small ranks of a result of \cite{KWW1} that, in matroid terminology, is equivalent to the following:

\begin{quote}
    {\it If $M$ is a simple regular matroid, then there exists a cocircuit whose deletion leaves $U_{1,1} \oplus N$ for some minor $N$ of rank $r(M) - 2$.}
\end{quote}

We call this result the {\it Coloop Splitting Theorem}. It has immediate applications to the structure of torus representations with the property that all isotropy groups are connected. This is due to the link found in \cite{KWW2} between such representations and regular matroids. That paper, along with the earlier paper \cite{KWW1}, proves structural results on torus representations to prove new obstructions to the existence of Riemannian metrics with positive curvature and isometric torus actions. We now describe an application of Theorem \ref{thm:BonusColoop} that improves a result in this area.

Using the Coloop Splitting Theorem, the paper \cite{KWW1} proves the Euler Characteristic Positivity Conjecture stated by Hopf in the 1930s for Riemannian metrics with positive sectional curvature that are invariant under a $T^5$-action. This was the first piece of evidence for this conjecture of this form where the rank of the torus did not need to grow to infinity in the manifold dimension. Later, Nienhaus’ refinement of the Coloop Splitting Theorem and other tools led to a proof that the conjecture holds more generally for $T^4$-invariant metrics as well \cite{Nienhaus-pre}.

Recently, the second author, Nienhaus, and Mouillé \cite{KennardMouilleNienhaus} used the Coloop Splitting Theorem to show a similar result where the curvature assumption is relaxed to positive second intermediate Ricci curvature but where the symmetry assumption is that the metric is $T^{10}$-invariant (see the corollary to \cite[Theorem A]{KennardMouilleNienhaus}). 
%Again, the Coloop Splitting Theorem played a key role.
Theorem \ref{thm:BonusColoop} implies that we can relax the symmetry assumption from $T^{10}$ to $T^9$. This also improves the first corollary to \cite[Theorem B]{KennardMouilleNienhaus}. 

\begin{main}
\label{thm:Ric_2}
    A $(4n)$-dimensional closed manifold admitting a $T^9$-invariant Riemannian metric with positive second intermediate Ricci curvature has positive Euler characteristic. If moreover the action has connected isotropy groups, then the manifold has vanishing odd Betti numbers.
\end{main}

Sections \ref{sec:Cographic} and \ref{sec:Graphic} are graph theoretic and contain the proofs of Theorems \ref{thm:BonusColoop-Cographic} and \ref{thm:BonusColoop-Graphic}, respectively. 
Section \ref{sec:Regular} extends these results to matroid theory, has the proof of Theorem \ref{thm:BonusColoop}, proves a corollary that feeds into the geometric application, and describes counterexamples and partial statements for smaller ranks. In Section \ref{sec:Geometry}, we explain applications to the structure of torus representations with connected isotropy groups and prove Theorem \ref{thm:Ric_2} as well as related new results in Riemannian geometry. We also explain the barrier to further reducing the symmetry assumption in Theorem \ref{thm:Ric_2}.

\subsection*{Acknowledgements} 
The second author was partially supported by NSF Grant DMS-2402129 and the Simons Foundation TSM program. The third author was partially supported by NSF Grant DMS-2402129 in Summer 2026.

\section{Loops from cycle contractions}
\label{sec:Cographic}

In this section, we prove Theorem \ref{thm:BonusColoop-Cographic}. The proof follows the idea of \cite[Theorem~4]{QiaoZhang10} with adjustments made to allow for non-simple graphs. We will adopt their aim of finding a cycle with two chords, and for this purpose we
will count loops and repeated edges whose vertices lie on a cycle as chords of that cycle.

\begin{proof}
[Proof of Theorem \ref{thm:BonusColoop-Cographic}]
We are given a $2$-vertex connected graph $G$ with minimum degree at least three, and we will show that there exists a cycle with at least two chords. If $G$ has only three vertices, then the proof follows easily. 
We may assume $|V(G)|\geq 4$.

Let $P = \{u_1, u_2,\dots, u_p\}$ be the vertex set of a maximal length path in $G$. Since $G$ has minimum degree at least three, $u_1$ must be incident to at least two edges not contained in $P$. By maximality of $P$, all neighbors of $u_1$ are in contained in $P$. Hence we have an edge $u_1 u_j$ for some $1 \leq j \leq p$ with the property that there is another edge $e$ between $u_1$ and some $u_i$ with $1 \leq i \leq j$. By taking $j$ to be maximal, we may further assume there is no edge between $u_1$ and $u_k$ with $k > j$.

By $2$-vertex connectivity,  $j \geq 3$ since otherwise $u_2$ is a cut-vertex. Moreover if $j = 3$, then, since $u_3$ is not a cut-vertex, $u_2$ is adjacent to a vertex besides $u_1$ and $u_3$. Moreover, that vertex is $u_k$ for some $k > j$ by the maximality of the length of $P$.

The cycle $\{u_1, \ldots, u_j\}$ already has one chord, namely $e$. 
If there is a second chord, we are done. 
Similarly, if $u_{j-1}$ is adjacent to $u_k$ for some $k > j$, then the cycle $\{u_1,\ldots, u_{j-1}, u_k, u_{k-1},\ldots, u_j\}$ has two chords, $e$ and $u_{j-1} u_j$. 
If neither of these cases occurs, then $u_1$ and $u_{j-1}$ are adjacent and of degree exactly three, and $j \geq 4$. Now $u_2$ is a vertex distinct from $u_1$ and $u_{j-1}$. It too has minimum degree three, so by maximality of $P$ again we have an edge of the form $u_2 u_k$ for some $k > j$. In this case, the cycle $\{u_2, \ldots, u_{j-1}, u_1, u_j, \ldots, u_k\}$ has two chords, $u_1 u_2$ and $u_{j-1} u_j$. 
\end{proof}

\section{Bridges from bond deletions}
\label{sec:Graphic}

In this section, we prove Theorem \ref{thm:BonusColoop-Graphic}. The statement has two parts, which correspond to Lemmas \ref{lem:Graphic_notKd} and \ref{lem:Graphic_Kd}. Together, these lemmas prove the theorem. We recall that a {\it bond} is a non-empty inclusion-minimal cut-set of edges. 

\begin{lemma}
\label{lem:Graphic_notKd}
    If $G$ is a simple $2$-vertex-connected graph that is not a complete graph, then there exists a bond $K$ such that $G\backslash K = H\oplus P_2$ for a subgraph $H$ of $G$.
\end{lemma}

\begin{proof}
    Set $E = E(G)$. Since $G$ is $2$-vertex-connected and not  $K_3$, it has at least four vertices. Since moreover $G$ is not a complete graph, there is a subset $S = \{u, v, w\}$ of three vertices such that $uv \in E$, $vw \in E$, and $uw \not\in E$. Among such choices of $S$ and choices of connected component $A \subseteq G \backslash S$, we fix $S$ and $A$ so that $|V(A)|$ is maximal.
    
    If $A$ is the only component of $G\backslash S$, then the set of edges $K$ between vertices of $A$ and $S$ is a bond. Further, $G\backslash K = A \oplus G[S] = H\oplus P_2$, so the theorem holds in this case. To finish the proof, we assume that $G \backslash S$ is disconnected and use maximality of $|V(A)|$ to derive a contradiction. 
    
    Set $B = G \backslash (S \cup A)$. By 2-vertex connectivity, at least two of $\{u,v,w\}$ are adjacent to some vertices in each component of $G\backslash S$. If $v$ is the only vertex adjacent to vertices in both $A$ and $B$, then $G-v$ is disconnected. Without loss of generality, $u$ is adjacent to vertices in both $A$ and $B$ and $\{v, w\}$ is adjacent to vertices in both $A$ and $B$.

    % We use maximality of $A$ to prove the following three claims.%:

    {\bf Claim 1}: If $s b \in E$ and $b b' \in E$ for some $s \in S$ and some $b, b' \in V(B)$, then $s b' \in E$.

    If $s b' \not\in E$, then the set $S' = \{s, b, b'\}$ induces a $P_2$. In addition, $G \backslash S'$ has a component that contains both $A$ and $\{v, w\}$ if $s = u$ and both $A$ and $u$ if $s \in \{v, w\}$. This contradicts maximality.

    {\bf Claim 2}: If $s b \in E$ and $s b' \in E$ for some $s \in S$ and some $b, b' \in V(B)$, then $b b' \in E$.

    The proof is similar to Claim 1 using the set $S ' = \{s, b, b'\}$ to derive a contradiction.

    We now use Claims 1 and 2 together with $2$-vertex connectivity and the maximality of $A$ to prove the following three claims.

    \textbf{Claim 3}: $B$ contains a vertex in $N(u)\cap N(w)$.

   Assume first that $w$ has no neighbors in $B$. By 2-vertex connectivity, $v$ is adjacent to some $b\in V(B)$, and thus the set $S'=\{b,v,w\}$ contradicts the choice of $S$. 
   By $2$-vertex connectivity, each component of $B$ is adjacent to at least two vertices in $S$. Therefore, any pair of components are adjacent to a common vertex $s \in S$, and by Claim 2, $B$ is connected.
   By Claim 1, $u$ and $w$ are adjacent to every vertex in $B$ and thus have a common neighbor in $B$.
    
    {\bf Claim 4}: There is no edge connecting $A$ to $v$.
    
    If there is an edge connecting $A$ to $v$, then the set $S' = \{u, b, w\}$ for any vertex $b$ in $B$ yields an induced $P_2$ by Claim 3 with the property that $G \backslash S'$ has a component containing both $A$ and $v$, which contradicts maximality.

    {\bf Claim 5}: There is a vertex $a$ in $A$ such that $A \backslash a$ is connected and $a \in N(u) \cup N(w)$.

    If $A \backslash a$ is connected for every vertex of $A$, then the claim clearly holds. We may assume $A$ has at least one cut-vertex, and we look at the block tree decomposition of $A$. We choose a cut-vertex $a$ and a block $A' \subseteq A$ that contains the cut-vertex $a$ but that has no other cut-vertices of $A$. This is accomplished by choosing a block that corresponds to a leaf in the block tree. Since $a$ is not a cut-vertex of $G$, there exists $a' \in V(A') - a$ connected to $u$ or to $w$. This proves the claim.

    We now finish the proof. After possibly relabeling $u$ and $w$, we may assume $a \in V(A)$ has the property that $A \backslash a$ is connected and $u a \in E$. If $A\backslash a$ is adjacent to $w$ in $G$, then $S' = \{a, u, b\}$ for any choice of $b \in V(B)$ has the property that $G[S'] = P_2$ and that $G \backslash S'$ has a component containing both $V(A) - a$ and $\{v, w\}$, which contradicts the maximality of $A$. Therefore we may assume that $V(A) \cap N(w) \subseteq \{a\}$. Combining Claim 4 and the $2$-vertex connectivity of $G$ implies that $A \cap N(w)$ is non-empty, so
        \[V(A) \cap N(w) = \{a\}.\]
    We now have $a w \in E$, and we can reverse the roles of $u$ and $w$ in this argument to conclude similarly that
        \[V(A) \cap N(u) = \{a\}.\]
    Hence $a$ is the only point in $A$ that connects to $S$. Since $G$ is $2$-vertex-connected, $|V(A)| = 1$.
    Then for any choice of $b\in B \cap N(u)$, which exists by Claim 3, the considering the set $S' = \{a,u,b\}$ again leads to a contradiction of the maximality of $V(A)$. 
    %Otherwise, the set $\{u,a,w\}$ contradicts the choice of $S$.
    We conclude that $A$ is the only component of $G\backslash S$.
    % This leads to a contradiction as in earlier claims, using the set $S' = \{u,b,w\}$ since, by a similar argument to the one in Claim 3, $v$ is adjacent to some vertex in $B$.
\end{proof}

\begin{lemma}
\label{lem:Graphic_Kd}
    If $G=K_d$ with $d \geq 3$, then for all $e\in E(G)$, there exists a bond $K$ such that $G\backslash K = H\oplus P_1$ for a subgraph $H$ of $G$ isomorphic to $K_{d-2}$, where $E(P_1) = \{e\}$.
\end{lemma}

\begin{proof}
    For any edge $e=uv$, the set of edges $K$ between $\{u, v\}$ and $G\backslash \{u,v\}$ forms a bond. Moreover, $G\backslash K = K_{d-2} \oplus P_1$ where $E(P_1) = \{e\}$.
\end{proof}

\section{Coloops from cocircuit deletions
}\label{sec:Regular}

In this section, we prove Theorem \ref{thm:BonusColoop} and Corollary \ref{cor:MinimalModularColoops} for matroids of rank at most six. The latter result is all we need later for the geometric applications. We also give an example showing that the conclusion of Corollary \ref{cor:MinimalModularColoops} no longer holds if we relax the rank condition to six to five but that results interpolating between our result in rank six and Nienhaus's in rank four do. We discuss how these barriers turn into roadblocks in the geometric applications in the next section.

Theorem \ref{thm:BonusColoop} is implied by the following, which includes a more explicit statement in Case (ii) that we use later.

\begin{theorem}
\label{thm:BonusColoop+}
    If $M$ is a simple regular matroid with $3 \leq r(M) \leq 6$, then one of the following holds:
    \begin{enumerate}
        \item[(i)] There is a cocircuit $C^*$ so that $M \backslash C^* = U_{2,2} \oplus N$ for some minor $N$ of $M$.
         \item[(ii)] $M$ is graphic on 
        $K_5$, $K_3 \oplus K_3$, 
        $K_6$, 
        $K_7$, $K_5 \oplus K_3$, or $K_3 \oplus K_3 \oplus K_3$, and for any $e \in E(M)$, there exists a cocircuit $C^*$ so that $M \backslash C^* = U_{1,1} \oplus N$ where $E(U_{1,1}) = \{e\}$ and $N$ is 
        a graphic matroid on $K_3$, $K_4$, $K_5$ or $K_3\oplus K_3$.
    \end{enumerate}
\end{theorem}

While not required for our applications, our proof shows Cases (i) and (ii) are disjoint.

\begin{proof}
We begin by reducing to the connected case. To do this, assume that 
$M = M_1 \oplus \dots \oplus M_n$ for some $n \geq 2$ where each $M_i$ is a connected, simple, regular matroid of rank at least one. 

\begin{enumerate}
\item
If $M_1 \neq M(K_d)$ for any $d$ or if $M_1 = M(K_4)$, then
the proof in the connected case gives a cocircuit $C^*$ of $M_1$ such that $M_1\backslash C^* = U_{2,2}\oplus N_1$ where $r(N_1) = r(M_1) - 3$. Further, $C^*$ is also a cocircuit of $M$ and
    \[M\backslash C^* = (U_{2,2}\oplus N_1) \oplus M_2 \oplus\ldots \oplus M_n
    = U_{2,2} \oplus N\]
for some minor $N$ of $M$ with  $r(N) = r(M) - 3$. Hence Conclusion (i) holds.

\item
If $M_1 = M(K_2)$ and $M_2$ has rank at least two, then the proof in the connected case implies that $M_2$ has a cocircuit whose deletion leaves a coloop. This cocircuit is also a cocircuit of $M$, and its deletion leaves two coloops. Hence Conclusion (i) holds. 

\item 
If $M_1 = M_2 = M(K_2)$, then $n \geq 3$ since $r(M) \geq 3$. Deleting any cocircuit from $M_3$ leaves two coloops in $M$, and Conclusion (i) holds.

\end{enumerate}

If none of these cases occurs, then up to relabeling we may assume that, for all $i$, $M_i = M(K_{d_i})$ with $d_i \not\in \{1, 2, 4\}$, where we note that $d_i \neq 1$ since each $M_i$ has rank at least one. It now follows from Theorem \ref{thm:BonusColoop-Graphic}.\eqref{thm:BonusColoop-Graphic_Kd} that for any $i$ and for any $e \in E(M_i)$, there is a cocircuit whose deletion leaves $e$ as a coloop together with $M(K_{d_i-2})$. Hence Conclusion (ii) holds. This completes the proof in the disconnected case assuming the connected case, so we may assume that $M$ is connected.

Suppose first that $M$ is cographic. Since duality preserves connectivity, $M^*$ is connected. Thus there is a 2-vertex-connected graph $G$ such that $M = M^*(G)$. Since $M$ is simple, it can have no circuits of size one or two. This implies $G$ has minimum degree at least three. By Theorem \ref{thm:BonusColoop-Cographic}, there is a cycle $C$ in $G$ with two chords. Further, the cycles of $G$ are cocircuits of $M^*(G)$, and $M^*(G) \backslash C = M^*(G/C)$ (see \cite[Theorem 3.1.1]{Oxley-book}). Therefore, the chords of $C$ in $G$ are loops in $G/C$, and $M^*(G/C)$ has at least two coloops, satisfying Conclusion (i).

Now suppose $M$ is graphic. Since $M$ is connected and simple, there is a simple 2-vertex-connected graph $G$ such that $M = M(G)$. 

If $G$ is not a complete graph, then Theorem \ref{thm:BonusColoop-Graphic}.\eqref{thm:BonusColoop-Graphic_notKd} implies that there is a bond $K$ such that $G\backslash K = H\oplus P_2$. Note that $K$ is a cocircuit in $M(G)$ and that 
    \[M(G)\backslash K = M(G\backslash K) = M(P_2) \oplus M(H) = U_{2,2} \oplus N,\]
where $N = M(H)$ is a minor of rank $r(M) - 3$, so Conclusion (i) holds in this case. 
If instead $G=K_d$ and if $d \not\in \{1, 2, 4\}$, then Theorem \ref{thm:BonusColoop-Graphic}.\eqref{thm:BonusColoop-Graphic_Kd} implies that, for any $e \in E(G)$, there exists a bond $K$ such that $G\backslash K = K_{d-2} \oplus P_1$ where $E(P_1) = \{e\}$. As in the previous case, $K$ is the cocircuit we seek, and Conclusion (ii) holds. 
If neither of these cases occurs, then $M = M(K_4)$. Since $K_4$ is self-dual, that is $M\cong M(K_4^*)$, Conclusion (i) follows by the cographic case.

Suppose third that $M\cong R_{10}$. The deletion of any one of the $15$ cocircuits of order six yields $U_{4,4}$. In particular, Conclusion (i) holds.

Finally suppose that $M$ is connected but not graphic, cographic, or $R_{10}$. By \cite{DanilovGrishukhin99,Israel-PhD} (see also \cite{DouthittIsraelKennard26}), $M$ has rank six and is a restriction of matroids we denote by $P_{12}$ or $R_{16}$, represented by the matrices $A_{12}$ and $A_{16}$ given in (\ref{equ:P12 representation}) and (\ref{equ:R16 representation}), respectively. We consider these in separate cases.
    
Suppose that $M = P_{12}|E$ for some $E\subseteq E(P_{12})$ with the following $\mathbb{Z}_2$ representation.%, where $P_{12}$ is a matroid represented by the following matrix over $\mathbb{Z}_2$.
    \begin{equation}\label{equ:P12 representation}
    A_{12} = \begin{blockarray}{ccccccccccccc}
    & f_1 & f_2 & p & e_1 & e_2 & e_3 & e_4 & e_5 & e_6 & e_7 & e_8 & e_9 \\
    \begin{block}{c[cccccccccccc]}
        v_0 & 1 & 1 & 0 & 0 & 0 & 0 & 0 & 0 & 0 & 0 & 0 & 0 \\
        v_1 & 1 & 0 & 1 & 0 & 0 & 0 & 0 & 1 & 1 & 0 & 0 & 1 \\
        v_2 & 0 & 0 & 0 & 1 & 0 & 0 & 0 & 1 & 1 & 1 & 0 & 0 \\
        v_3 & 0 & 0 & 0 & 0 & 1 & 0 & 0 & 0 & 1 & 1 & 1 & 0 \\
        v_4 & 0 & 0 & 0 & 0 & 0 & 1 & 0 & 0 & 0 & 1 & 1 & 1 \\
        v_5 & 0 & 0 & 0 & 0 & 0 & 0 & 1 & 1 & 0 & 0 & 1 & 1 \\
    \end{block}\end{blockarray}
    \end{equation}
    Recall $P_{12} = P_{K_2}(M(K_3), R_{10})$, and, for each element $e$ in $E(R_{10})$, the matroid $R_{10}\backslash e$ is graphic. Thus, for any $1\leq i \leq 9$, if the element $e_i$ is not present in $M$, then $M$ is isomorphic to a restriction of $P_{K_2}(M(K_3), R_{10}\backslash e)$, which is graphic. Hence $M$ is also graphic, and the result follows. Thus $\{e_1,e_2,\dots,e_9\} \subseteq E$. Moreover $\{f_1,f_2\}\subseteq E$ because $M$ is connected and has rank six. 
    Let $C^*$ be the set corresponding to the support of $v_1+v_3$. Since $C^*$ is the support of a vector in the row space, $C^*$ is a cocircuit if and only if $M\backslash C^*$ has rank 5. The matroid $M\backslash C^*$ is isomorphic to $U_{5,5}$, having ground set $\{f_2,e_1,e_3,e_4,e_6\}$. Thus $C^*$ is a cocircuit of $M$, satisfying Conclusion~(i).

    Now suppose $M = R_{16}|E$ for some $E\subseteq E(R_{16})$ with the following $\mathbb{Z}_2$ representation.%. The matroid $R_{16}$ is represented over $\mathbb{Z}_2$ by the matrix $A_{16}$ below. 
    \begin{equation}\label{equ:R16 representation}
    A_{16} = \begin{blockarray}{ccccccccccccccccc}
    & e_{1} & e_2 & e_3 & e_4 & e_5 & e_6 & f_1 & f_2 & f_3 & f_{4} & f_{5} & f_{6} & f_{7} & f_{8} & f_{9} & f_{10} \\
    \begin{block}{c[cccccccccccccccc]}
        v_1 & 1 & 0 & 0 & 0 & 0 & 0 & 1 & 1 & 1 & 0 & 0 & 0 & 1 & 1 & 0 & 0 \\
        v_2 & 0 & 1 & 0 & 0 & 0 & 0 & 1 & 1 & 0 & 1 & 0 & 0 & 1 & 1 & 0 & 0 \\
        v_3 & 0 & 0 & 1 & 0 & 0 & 0 & 1 & 0 & 0 & 0 & 1 & 0 & 0 & 0 & 0 & 1 \\
        v_4 & 0 & 0 & 0 & 1 & 0 & 0 & 0 & 1 & 0 & 0 & 0 & 1 & 0 & 0 & 0 & 1 \\
        v_5 & 0 & 0 & 0 & 0 & 1 & 0 & 0 & 0 & 1 & 0 & 1 & 1 & 0 & 1 & 1 & 0 \\
        v_6 & 0 & 0 & 0 & 0 & 0 & 1 & 0 & 0 & 0 & 1 & 1 & 1 & 0 & 1 & 1 & 0 \\
    \end{block}\end{blockarray}
    \end{equation}
    Since $M$ is not graphic or cographic, we know from \cite{DanilovGrishukhin99,Israel-PhD,DouthittIsraelKennard26} that $E$ contains the first $12$ vectors and some subset of the last four. 
    We consider the subset $C^*$ given by the support of the row space vector $v_1 + v_3 + v_6$. We have that $E(M \backslash C^*)$ equals $\{e_2, e_4, e_5, f_1, f_{5}\}$ together possibly with $f_{8}$. Note that $M \backslash C^*$ has rank five, so $C^*$ is a cocircuit. Moreover, at least $e_2$, $e_4$, and $e_5$ are coloops, so $M \backslash C^*$ is either $U_{5,5}$ or $U_{3,3} \oplus M(K_3)$. Hence $M$ satisfies Conclusion (i).
\end{proof}

For the purposes of our application in the next section, we explore applications of Theorem \ref{thm:BonusColoop} for the case of weighted matroids.

First, it is clear from Theorem \ref{thm:BonusColoop} that, for a simple regular weighted matroid $M$ of rank at most six and weight function $\mu \colon E(M) \to \mathbb R$, there exists a cocircuit $C^*$ whose deletion either leaves two coloops or whose deletion leaves a coloop that has minimum weight among elements that remain. In rank four, this is how Nienhaus showed the existence of a cocircuit whose deletion leaves a coloop of minimum weight among elements that remain. In rank six, we can apply Theorem \ref{thm:BonusColoop} again to the minor of rank $r(M) - 3$ in the first case, if needed, and conclude the following: {\it A simple regular weighted matroid of rank six has a pair of cocircuits whose deletion leaves a coloop of minimum weight}. In general, one might ask how many coloops one needs to delete to guarantee such a {\it locally minimal coloop splitting}. 

For our purposes, we go in a different direction and focus on integer-valued weight functions. The following is the precise, and the only, statement from this section that we use in the next section.

\begin{corollary}\label{cor:MinimalModularColoops}
    Let $M$ be a simple regular matroid of rank six and $\mu\colon E(M) \to \Z_+$ be a positive, integer-valued weight function. There exists an induced restriction $N$ of $M$ with rank four and disjoint cocircuits $C_1^*$ and  $C_2^*$ of $N$ such that the following hold:
    \begin{enumerate}
        \item[(i)] $3\mu(C_1^*) \leq \mu(N\backslash C_2^*)$ and
        \item[(ii)] $\mu(C_1^*) \equiv \mu(N\backslash C_2^*) \bmod 2$.
    \end{enumerate}
\end{corollary}

\begin{proof}
    First, we prove the result when $U_{4,4}$ is an induced restriction of $M$. %there are two cocircuits of $M$ whose deletion leaves a $U_{4,4}$. 
    In this case, we can set $N = U_{4,4}$, choose $C_1^*$ to be the minimum weight coloop of $N$, and choose $C_2^*$ so that the sum of the third and fourth coloops is even. This is possible since $N \backslash C_1^* = U_{3,3}$, which means at least two of these three coloops have the same weight modulo two.
    
    Second, we prove that $U_{4,4}$ arises as an induced restriction if there exists a cocircuit $C^*$ that splits off $U_{2,2}$ as in Conclusion (i) of Theorem \ref{thm:BonusColoop}. Indeed, what remains is $U_{2,2} \oplus R$ where $R$ has rank three. Applying the Theorem \ref{thm:BonusColoop+} to $R$ yields a $U_{2,2}$ splitting of $R$. Hence we obtain $U_{4,4}$ by deleting two cocircuits. %altogether we have four coloops after deleting two cocircuits. 
    This shows the claim.

    We may therefore assume that $M$ does not admit a $U_{2,2}$ splitting. We fix $e_1 \in E(M)$ having minimum weight, and we apply Theorem \ref{thm:BonusColoop} to find a cocircuit $D_1^*$ such that
        \[M \backslash D_1^* = U_{1,1} \oplus M',\]
    where $E(U_{1,1}) = \{e_1\}$ and where $M'$ is a rank-four minor of $M$. Moreover, by \ref{thm:BonusColoop+}, $M'$ is the cycle matroid on $K_5$ or $K_3\oplus K_3$.
    
    We fix $e_2$ so that $\mu(e_2)$ is odd if possible, and otherwise we fix any $e_2 \in E(M')$. By Theorem \ref{thm:BonusColoop}, we can delete a cocircuit $D_2^*$ of $M'$ to obtain a splitting 
        \[M' \backslash D_2^* = U_{1,1} \oplus M'',\]
    where $E(U_{1,1}) = \{e_2\}$ and where $M''$ is a rank-two minor of $M'$ and therefore of $M$. Now consider the matroid
        \[N = (M \backslash D_1^*) \backslash D_2^* = U_{2,2} \oplus M'',\]
    where $E(U_{2,2}) = \{e_1, e_2\}$. Moreover, by Theorem~\ref{thm:BonusColoop+}, $M''$ is the cycle matroid of $K_3$. %\cong M(K_3)$ since otherwise we have a $U_{4,4}$ splitting.
    
    If $\mu(M'')$ is even, then setting $C_1^* = \{e_1\}$ and $C_2^* = \{e_2\}$ gives us a solution to our problem. If $\mu(M'')$ is odd, then some element $e'' \in E(M'')\subseteq E(M')$ has odd weight. Our choice of $e_2$ implies $\mu(e_2)$ is also odd. 
    %Since $M'' = M(K_3)$, the set $E(M'') - e''$ is a cocircuit $C_2^*$. Together with $C_1^* = \{e_1\}$, we have a solution to our problem since $C_1^*$ and $C_2^*$ are disjoint and 
    If $C_1^*=\{e_1\}$ and $C_2^*=E(M'')-e''$, then 
        \[\mu\of{N \backslash (C_1^* \oplus C_2^*)}
        = \mu(e_2) + \mu(e'') \equiv 0 \bmod 2,\]
    and hence $C_1^*$ and $C_2^*$ are disjoint cocircuits of $N$ that satisfy the claim.
\end{proof}

The following example shows that in rank five, Corollary~\ref{cor:MinimalModularColoops} fails.
\begin{example}
    Let $M$ be the cycle matroid of $K_6$. Fix a vertex $v$ of $K_6$ and let $\mu$ be the weight function such that $\mu(e)=2$ for all edges $e$ incident to $v$ and $\mu(e)=1$ otherwise. 
    In this weighted matroid, there are, up to isomorphism, six weighted matroids that arise as induced restrictions of $(M,\mu)$ with rank four or greater. In particular, we have one weight on $M(K_6)$, two weights on $M(K_5)$, one weight on $M(K_3) \oplus M(K_3)$, and two weights on $M(K_2)\oplus M(K_4)$. 
    
    If $N$ is $M(K_6)$ or $M(K_5)$, then $N$ has no two disjoint cocircuits and the result fails. 
    
    If $N$ is $M(K_3)\oplus M(K_3)$ then $\mu(N)=8$ and one may take a cocircuit $C_1^*$ with $\mu(C_1^*)=2$, but $\mu(C_2^*)\geq 3$ for all cocircuits $C_2^*$ that are disjoint from $C_1^*$, so the estimate $3\mu(C_1^*) \leq \mu(N \backslash C_2^*)$ fails for any such choice. 
    
    If $N$ is $M(K_2)\oplus M(K_4)$, then let $e=E(M(K_2))$. Thus either $\mu(e)=2$ or $\mu(e)=1$. Suppose first that $\mu(e)=2$. If $C_1^*\subseteq E(M(K_4))$ then $C_2^*=\{e\}$ and $3\leq \mu(C_1^*)$ while $\mu(N\backslash C_2^*)=6$, and Condition (i) fails. Therefore $C_1^*=\{e\}$ and for any cocircuit $C_2^*$ disjoint from $C_1^*$, we have $\mu(N\backslash C_2^*)\leq 5$ and Condition (i) fails. Now suppose $\mu(e)=1$. Again if $C_1^*\subseteq E(M(K_4))$, then $C_2^*=\{e\}$ and $4\leq\mu(C_1)$ while $\mu(N\backslash C_2^*)=9$. Instead let $C_1^* = \{e\}$. Up to symmetry, there are three options for $C_2^*$, that is, a four edge cut with weight six, a three edge cut with weight four, or a three edge cut with weight six. In any case, $\mu(N\backslash C_2^*) \equiv 0 \bmod 2$, so Condition (ii) fails.
\end{example}

While the above example shows that both Conditions (i) and (ii) of Corollary~\ref{cor:MinimalModularColoops} need not hold in rank five, the next result shows that Condition (i) still holds.

\begin{theorem}\label{thm:rank-5 cocircuits}
    If $M$ is a simple regular matroid of rank five and if $\mu$ is a real-valued function on $E(M)$, then there is an induced restriction $N$ of $M$ with rank four and disjoint cocircuits $C_1^*$ and $C_2^*$ such that $3\mu(C_1^*)\leq \mu(N\backslash C_2^*)$.
\end{theorem}

\begin{proof}
Fix a weighted matroid $(M, \mu)$ as in the theorem. We prove this result in four special cases, which uses graphs $G_1,G_2,G_3,$ and $G_4$ from Figure~\ref{fig:six-vertex grpahs}. We prove in Lemma \ref{lem:LocallyMinimalColoop} below that one of these cases occurs. 

\vspace{1em}
{\bf Case 1}: 
Suppose there exists a cocircuit $C^*$ such that $M \backslash C^*$ has a coloop $e$ with the property that $\mu(e) \leq \mu(f)$ for all $f \in E(M \backslash C^*)$. If $M\backslash C^* = U_{4,4}$, the result holds immediately by taking $C_1^*$ and $C_2^*$ to be the two lightest elements of $M\backslash C^*$. Otherwise $M \backslash C^* = U_{1,1} \oplus N$ where $E(U_{1,1}) = \{e\}$ and where $N$ is a rank-three minor on at least four elements. Let $C_1^* = \{e\}$ and choose a cocircuit $C_2^*$ in $N$ satisfying $\mu(C_2^*) \leq \tfrac 1 2 \mu(N)$ (using, for example, \cite[Theorem D]{DouthittIsraelKennard26}). Hence $3 \mu(C_1^*) + \mu(C_2^*) \leq \mu(U_{1,1} \oplus N)$, as needed.

\vspace{1em}
{\bf Case 2}: 
$M$ is $M(G_1)$ or $M(G_2)$, where $G_1$ and $G_2$ are the two possible $1$-sums of $K_3$ and $K_4\backslash e$. Both graphs yield isomorphic cycle matroids. Let $f$ be the edge adjacent to the two degree-three vertices of $K_4\backslash e$. If an edge $h$ of $K_3$ has minimum weight in $M\backslash f$, there is a cocircuit deletion leaving $h$ as a bridge. Take $C_1^* = \{h\}$ and $C_2^*\subseteq E(M(K_4\backslash e))$ to be any cocircuit that contains $f$. Then $3\mu(C_1^*) \leq \mu(N\backslash C_2^*)$. On the other hand, if the minimum-weight element of $M\backslash f$ is an edge $h$ of $K_4\backslash e$, there is a cocircuit deletion which leaves $h$ and $g$ as bridges where $g$ is another edge of $K_4\backslash e$. Taking $C_1^* = \{h\}$ and $C_2^* = \{g\}$ gives the desired conclusion.

\vspace{1em}
{\bf Case 3}: 
$M = M(G_3)$, where $G_3$ is the $6$-cycle on vertices $1,\ldots,6$ with additional edges $24$, $46$, and $62$. Among the edges of this $6$-cycle, up to relabeling of the vertices, the edge $12$ has the lowest weight. There is a cocircuit deletion leaving $12$ and $23$ as bridges and the triangle connecting $4$, $5$, and $6$. Taking $C_1^* = \{12\}$ and $C_2^* = \{23\}$ solves our problem, since
\[3\mu(C_1^*) \leq \mu(12) + \mu(45)+\mu(56) \leq \mu(N\backslash C_2^*).\]

\vspace{1em}
{\bf Case 4}: 
$M = M(G_4)$, where $G_4 = K_6 \backslash F$ and where $F$ is the edge set of a subgraph isomorphic to $K_{1,3}$. Let $u$ be the vertex of $K_6 \backslash F$ that is adjacent to only two other vertices, $v_1$ and $v_2$. Let $b_1, \dots, b_6$ be the edges incident to exactly one of $v_1$ or $v_2$ and not incident to $u$. Up to relabeling, $b_1$ is incident to $v_1$ and has the lowest weight out of all $b_i$. There is a cocircuit deletion leaving $uv_1$ and $b_1$ as bridges and a triangle containing some $b_j$, $b_k$, and one unlabeled edge. Taking $C_1^* = \{b_1\}$ and $C_2^* = \{uv_1\}$ solves our problem, since
$3\mu(C_1^*) \leq \mu(b_1) + \mu(b_j)+\mu(b_k) \leq \mu(N\backslash C_2^*).$
\end{proof}

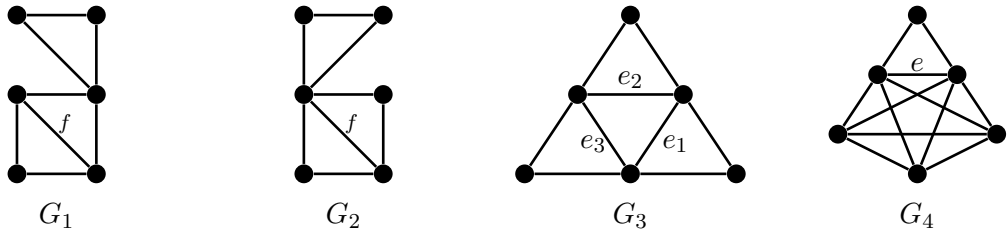
\begin{figure}[h]
\centering
\begin{subfigure}{0.24\textwidth}
\centering
    \begin{tikzpicture}[scale=.35]
        \node[circle,fill=black,minimum size=3pt, inner sep=2.5pt] (v1) at (0,0) {};
        \node[circle,fill=black,minimum size=3pt, inner sep=2.5pt] (v2) at (3,0) {};
        \node[circle,fill=black,minimum size=3pt, inner sep=2.5pt] (v3) at (0,6) {};
        \node[circle,fill=black,minimum size=3pt, inner sep=2.5pt] (v4) at (0,3) {};
        \node[circle,fill=black,minimum size=3pt, inner sep=2.5pt] (v5) at (3,3) {};
        \node[circle,fill=black,minimum size=3pt, inner sep=2.5pt] (v6) at (3,6) {};
        \draw[line width=1pt] (v1) -- (v2);
        \draw[line width=1pt] (v1) -- (v4);
        \draw[line width=1pt] (v6) -- (v3);
        \draw[line width=1pt] (v6) -- (v5);
        \draw[line width=1pt] (v3) -- (v5);
        \draw[line width=1pt] (v2) -- (v4) node[midway, above, xshift=3pt,yshift=-4pt] {\tiny $f$};
        \draw[line width=1pt] (v2) -- (v5);
        \draw[line width=1pt] (v4) -- (v5);
    \end{tikzpicture}
    \subcaption*{$G_1$}
    \end{subfigure}
    \begin{subfigure}{0.24\textwidth}
    \centering
    \begin{tikzpicture}[scale=.35]
        \node[circle,fill=black,minimum size=3pt, inner sep=2.5pt] (v1) at (0,0) {};
        \node[circle,fill=black,minimum size=3pt, inner sep=2.5pt] (v2) at (3,0) {};
        \node[circle,fill=black,minimum size=3pt, inner sep=2.5pt] (v3) at (0,6) {};
        \node[circle,fill=black,minimum size=3pt, inner sep=2.5pt] (v4) at (0,3) {};
        \node[circle,fill=black,minimum size=3pt, inner sep=2.5pt] (v5) at (3,3) {};
        \node[circle,fill=black,minimum size=3pt, inner sep=2.5pt] (v6) at (3,6) {};
        
        \draw[line width=1pt] (v1) -- (v2);
        \draw[line width=1pt] (v1) -- (v4);
        \draw[line width=1pt] (v6) -- (v3);
        \draw[line width=1pt] (v6) -- (v4);
        \draw[line width=1pt] (v3) -- (v4);
        \draw[line width=1pt] (v2) -- (v4) node[midway, above, xshift=3pt,yshift=-4pt] {\tiny $f$};
        \draw[line width=1pt] (v2) -- (v5);
        \draw[line width=1pt] (v4) -- (v5);
    \end{tikzpicture}
    \subcaption*{$G_2$}
    \end{subfigure}
\begin{subfigure}{0.24\textwidth}
\centering
    \begin{tikzpicture}[scale=.35]
        \node[circle,fill=black,minimum size=3pt, inner sep=2.5pt] (v1) at (0,0) {};
        \node[circle,fill=black,minimum size=3pt, inner sep=2.5pt] (v2) at (4,0) {};
        \node[circle,fill=black,minimum size=3pt, inner sep=2.5pt] (v3) at (8,0) {};
        \node[circle,fill=black,minimum size=3pt, inner sep=2.5pt] (v4) at (2,3) {};
        \node[circle,fill=black,minimum size=3pt, inner sep=2.5pt] (v5) at (6,3) {};
        \node[circle,fill=black,minimum size=3pt, inner sep=2.5pt] (v6) at (4,6) {};
        \draw[line width=1pt] (v1) -- (v2);
        \draw[line width=1pt] (v1) -- (v4);
        \draw[line width=1pt] (v2) -- (v3);
        \draw[line width=1pt] (v3) -- (v5);
        \draw[line width=1pt] (v2) -- (v4) node[midway, below, xshift=-4pt,yshift=4pt] {\small $e_3$};
        \draw[line width=1pt] (v2) -- (v5) node[midway, below,xshift=7pt,yshift=4pt] {\small $e_{1}$};
        \draw[line width=1pt] (v4) -- (v5) node[midway, above,yshift=-2pt] {\small $e_{2}$};
        \draw[line width=1pt] (v5) -- (v6);
        \draw[line width=1pt] (v4) -- (v6);
    \end{tikzpicture}
    \subcaption*{$G_3$}
    \end{subfigure}
    \begin{subfigure}{0.24\textwidth}
    \centering
    \begin{tikzpicture}[scale=.35]
        \node[circle,fill=black,minimum size=3pt, inner sep=2.5pt] (v1) at (0,.75) {};
        \node[circle,fill=black,minimum size=3pt, inner sep=2.5pt] (v2) at (3,-.75) {};
        \node[circle,fill=black,minimum size=3pt, inner sep=2.5pt] (v3) at (6,.75) {};
        \node[circle,fill=black,minimum size=3pt, inner sep=2.5pt] (v4) at (1.5,3) {};
        \node[circle,fill=black,minimum size=3pt, inner sep=2.5pt] (v5) at (4.5,3) {};
        \node[circle,fill=black,minimum size=3pt, inner sep=2.5pt] (v6) at (3,5.25) {};
        \draw[line width=1pt] (v1) -- (v2);
        \draw[line width=1pt] (v1) -- (v3);
        \draw[line width=1pt] (v1) -- (v5);
        \draw[line width=1pt] (v3) -- (v4);
        \draw[line width=1pt] (v1) -- (v4);
        \draw[line width=1pt] (v2) -- (v3);
        \draw[line width=1pt] (v3) -- (v5);
        \draw[line width=1pt] (v2) -- (v4);
        \draw[line width=1pt] (v2) -- (v5);
        \draw[line width=1pt] (v4) -- (v5) node[midway, above, yshift=-3pt] {\small $e$};
        \draw[line width=1pt] (v5) -- (v6);
        \draw[line width=1pt] (v4) -- (v6);
    \end{tikzpicture}
    \subcaption*{$G_4$}
    \end{subfigure}
    \caption{Six-vertex graphs whose graphic matroids have no locally minimal coloop splitting} \label{fig:six-vertex grpahs}
\end{figure}

To finish the proof of Theorem~\ref{thm:rank-5 cocircuits}, we show the following:

\begin{lemma}[Locally minimal coloop splitting]
\label{lem:LocallyMinimalColoop}
    If $M$ is a simple regular matroid of rank between two and five that is not isomorphic to $M(G_i)$ for some $i \in \{2, 3, 4\}$, then, for any real-valued weight function on $E(M)$, there exists a cocircuit $C^*$ such that $M \backslash C^*$ has a coloop $e$ with the property that $e$ has minimum weight in $E(M \backslash C^*)$.
\end{lemma}

While it is not needed for the proof of Theorem \ref{thm:rank-5 cocircuits}, we note that the converse holds. For example, for $M(G_1)$, the weight function equal to $1$ on the edge labeled $f$, to $3$ on the edges of $(K_4 \backslash e)\backslash f$, and to $2$ on the edges in $K_3$, there is no locally minimal coloop splitting. Indeed, no bond leaves $f$ as a bridge, and any bond that contains $f$ leaves all of $K_3$. Thus $M$ does have a locally minimal coloop splitting.% not leave a bridge of minimal weight.

Similarly, for $M(G_3)$, we take a weight function equal to $1$ on the edges $e_1$, $e_2$, and $e_3$ and equal to $2$ on the remaining edges. The only edges that become bridges after deleting a bond are weight two, and every bond deletion also leaves a weight-one edge.

For $M(G_4)$, one can similarly check that putting weight $1$ on $e$, weight $2$ on the edges of the triangle connecting the degree-four vertices, and weight $3$ on all other edges produces a weighted matroid with no locally minimal coloop splitting.

\begin{proof}
    If $M$ has rank two or three, then $M = M(G)$ for some graph $G$ on three or four vertices. Unless $G = K_4 \backslash e$, we have the stronger property that any edge in $G$ can become a bridge after deleting some bond. This {\it globally minimal coloop splitting property} both implies the required property and is used later in the proof. For $K_4 \backslash e$, the edge $f$ vertex disjoint from $e$ might be the unique edge of minimum weight, but deleting bonds that contain $f$ allow us to realize any of the other edges as a bridge of the deletion. Therefore, the locally minimal property holds for $M(K_4 \backslash e)$.
    
    If $M$ is disconnected, then $M$ is a direct sum, $M_1 \oplus M_2$, of simple regular matroids of positive rank. If $M_1$ is $U_{k,k}$ for some $k \in \{1,2\}$, then by induction on the rank we can produce a locally minimal coloop splitting in $M_2$. Since the elements in $U_{k,k}$ remain coloops as well, we have in any case split off a coloop of minimum weight. If $M_1$ is not of this form, then for rank reasons we can relabel so that $M_1 = M(K_3)$ and $r(M_2) \in \{2, 3\}$. Note that $M_1$ has the globally minimal coloop splitting property. If $M_2$ does as well, then clearly $M_1 \oplus M_2$ has the globally minimal coloop splitting property and we are done. Otherwise $M_2 = M(K_4 \backslash e)$ as shown in the proof in rank three. In this case, $M$ is isomorphic to $M(G_2)$, so the proof is complete.

    We may now assume that $M$ is connected and has rank four or five. We have that $M$ is either graphic or cographic on a $2$-vertex-connected graph $G$ or equal to the sporadic matroid $R_{10}$ (see \cite{Seymour80,DouthittIsraelKennard26, DanilovGrishukhin99}). In the case of $M = R_{10}$, we have a $U_{4,4}$ splitting after deleting any of the $6$-element cocircuits, so the result holds.
    
    Suppose first that $M$ is cographic but not graphic. We may assume that 
        \[M = M^*(G)\backslash F = M^*(G/F)\] 
    for some $3$-edge-connected, cubic graph $G$ with Betti number $4$ or $5$ and some subset $F$ of edges with the property that $G/F$ is non-planar (see \cite[Lemma 1.3]{DouthittIsraelKennard26}). Note that $G$ is one of the graphs shown in Figure \ref{fig:K33_G53_G54}. Our strategy is to show that $M$ has a cocircuit whose deletion leaves a $U_{r-1, r-1}$ where $r = r(M)$. We accomplish this by showing that $G/F$ has a Hamiltonian cycle $C$, as then every edge of $(G/F)/C$ is a loop. This then guarantees that choosing the cocircuit of $M^*(G)$ corresponding to $C$ leaves a minimal weight coloop. It is quick to check that $K_{3,3}$ and $G_{5,g}$ are Hamiltonian for $g\in \{3,4\}$, so it suffices to show that the same holds for all contractions $G_{5,g}/F$ that are non-planar and not isomorphic to $K_{3,3}$.

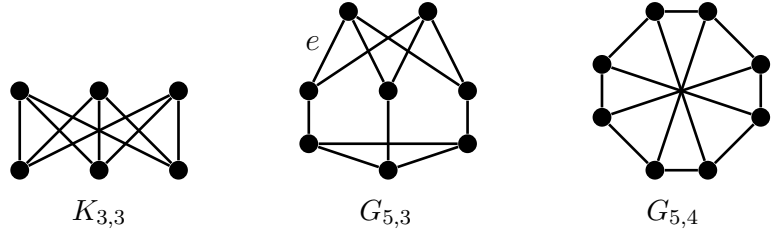
\begin{figure}[ht]
\centering
\begin{subfigure}{0.24\textwidth}
\centering
    \begin{tikzpicture}[scale=.35]
        \node[circle,fill=black,minimum size=3pt, inner sep=2.5pt] (v1) at (0,0) {};
        \node[circle,fill=black,minimum size=3pt, inner sep=2.5pt] (v2) at (3,0) {};
        \node[circle,fill=black,minimum size=3pt, inner sep=2.5pt] (v3) at (6,0) {};
        \node[circle,fill=black,minimum size=3pt, inner sep=2.5pt] (v4) at (0,3) {};
        \node[circle,fill=black,minimum size=3pt, inner sep=2.5pt] (v5) at (3,3) {};
        \node[circle,fill=black,minimum size=3pt, inner sep=2.5pt] (v6) at (6,3) {};
        \draw[line width=1pt] (v1) -- (v4);
        \draw[line width=1pt] (v1) -- (v5);
        \draw[line width=1pt] (v1) -- (v6);
        \draw[line width=1pt] (v2) -- (v4);
        \draw[line width=1pt] (v2) -- (v5);
        \draw[line width=1pt] (v2) -- (v6);
        \draw[line width=1pt] (v3) -- (v4);
        \draw[line width=1pt] (v3) -- (v5);
        \draw[line width=1pt] (v3) -- (v6);
        
    \end{tikzpicture}
    \subcaption*{$K_{3,3}$}
    \end{subfigure}
    \begin{subfigure}{0.24\textwidth}
    \centering
    \begin{tikzpicture}[scale=.35]
        \node[circle,fill=black,minimum size=3pt, inner sep=2.5pt] (v1) at (0,1) {};
        \node[circle,fill=black,minimum size=3pt, inner sep=2.5pt] (v2) at (3,0) {};
        \node[circle,fill=black,minimum size=3pt, inner sep=2.5pt] (v3) at (6,1) {};
        \node[circle,fill=black,minimum size=3pt, inner sep=2.5pt] (v4) at (0,3) {};
        \node[circle,fill=black,minimum size=3pt, inner sep=2.5pt] (v5) at (3,3) {};
        \node[circle,fill=black,minimum size=3pt, inner sep=2.5pt] (v6) at (6,3) {};
        \node[circle,fill=black,minimum size=3pt, inner sep=2.5pt] (v7) at (1.5,6) {};
        \node[circle,fill=black,minimum size=3pt, inner sep=2.5pt] (v8) at (4.5,6) {};
        \draw[line width=1pt] (v1) -- (v2);
        \draw[line width=1pt] (v1) -- (v3);
        \draw[line width=1pt] (v2) -- (v3);
        \draw[line width=1pt] (v1) -- (v4);
        \draw[line width=1pt] (v2) -- (v5);
        \draw[line width=1pt] (v3) -- (v6);
        \draw[line width=1pt] (v4) -- (v7) node[midway, above, xshift=-6pt,yshift=-4pt] {$e$};
        \draw[line width=1pt] (v4) -- (v8);
        \draw[line width=1pt] (v5) -- (v7);
        \draw[line width=1pt] (v5) -- (v8);
        \draw[line width=1pt] (v6) -- (v7);
        \draw[line width=1pt] (v6) -- (v8);
        
    \end{tikzpicture}
    \subcaption*{$G_{5,3}$}
    \end{subfigure}
    \begin{subfigure}{0.24\textwidth}
    \centering
    \begin{tikzpicture}[scale=.35]
        \node[circle,fill=black,minimum size=3pt, inner sep=2.5pt] (v1) at (0,0) {};
        \node[circle,fill=black,minimum size=3pt, inner sep=2.5pt] (v2) at (2,-2) {};
        \node[circle,fill=black,minimum size=3pt, inner sep=2.5pt] (v3) at (4,-2) {};
        \node[circle,fill=black,minimum size=3pt, inner sep=2.5pt] (v4) at (6,0) {};
        \node[circle,fill=black,minimum size=3pt, inner sep=2.5pt] (v5) at (0,2) {};
        \node[circle,fill=black,minimum size=3pt, inner sep=2.5pt] (v6) at (2,4) {};
        \node[circle,fill=black,minimum size=3pt, inner sep=2.5pt] (v7) at (4,4) {};
        \node[circle,fill=black,minimum size=3pt, inner sep=2.5pt] (v8) at (6,2) {};
        \draw[line width=1pt] (v1) -- (v2);
        \draw[line width=1pt] (v2) -- (v3);
        \draw[line width=1pt] (v3) -- (v4);
        \draw[line width=1pt] (v5) -- (v6);
        \draw[line width=1pt] (v6) -- (v7);
        \draw[line width=1pt] (v7) -- (v8);
        \draw[line width=1pt] (v1) -- (v5) node[midway, above, xshift=-6pt,yshift=-7pt] {};
        \draw[line width=1pt] (v2) -- (v7);
        \draw[line width=1pt] (v3) -- (v6);
        \draw[line width=1pt] (v4) -- (v8);
        \draw[line width=1pt] (v1) -- (v8);
        \draw[line width=1pt] (v4) -- (v5);
        
    \end{tikzpicture}
    \subcaption*{$G_{5,4}$}
    \end{subfigure}
    \caption{Maximal Nonplanar Graphs}
    \label{fig:K33_G53_G54}
\end{figure}

    Consider first graphs of the form $G_{5,3}/F$. 
    Because it is non-planar, $F$ does not contain an edge in the same edge orbit as $e$. 
    Similarly, $F$ has at most two edges since otherwise $G_{5,3}/F$ is planar or $K_{3,3}$. 
    This leaves six cases up to symmetry, and in each case one can see that there is a Hamiltonian cycle in $G_{5,3}$ containing $F$ and hence one in $G_{5,3}/F$. This completes the proof in this case.

    Now consider $G_{5,4}/F$. By non-planarity, $F$ does not contain any of the chords of the outer $8$-cycle. Therefore, the outer $8$-cycle contains $F$ and hence contracts to a Hamiltonian cycle in $G_{5,4}/F$ with four chords, completing the proof.

    To finish the proof, we may assume that $M$ is the connected graphic matroid $M(G)$ for some simple, $2$-vertex-connected graph $G$ on five or six vertices. Moreover, we may assume $M(G)$ does not satisfy the globally minimal coloop splitting property, so there exists an edge $e = uv$
    of $G$ such that there is no bond $K$ where $e$ is a bridge in $G\backslash K$. In particular, $G\backslash \{u,v\}$ is disconnected. We complete the proof by showing that there is a cocircuit of $M$ whose deletion yields a $U_{r-1,r-1}$, where $r$ is the rank of $M$. This is equivalent to finding a bond of $G$ whose deletion leaves a forest.

    First suppose $G\backslash \{u,v\}$ has two components, $A_1$ and $A_2$, that have two vertices each. We denote the vertex sets by $V(A_i) = \{x_i, y_i\}$ for $i\in\{1,2\}$. By $2$-vertex-connectedness of $G$, both endpoints of $e$ must be adjacent to both components of $G\backslash \{u,v\}$. After relabeling, we may assume $ux_i$ and $vy_i$ are edges for $i\in \{1,2\}$. Regardless of which other edges may be present in $G$, the bond separating $\{u,x_1,x_2\}$ and $\{v,y_1,y_2\}$ leaves a forest with two components.
    
    Now suppose $G\backslash \{u,v\}$ has a component with a single vertex $w$. Since $G$ is $2$-vertex connected, $w$ forms a triangle with $u$ and $v$. Because $G$ is $2$-vertex-connected, there are vertices $x$ and $z$ not in the triangle $\{u, v, w\}$ such that $ux$ and $vz$ are edges. Moreover, if $|V(G)| = 5$, then the cut-set $K_{ux}$ of edges incident to $ux$ is a bond whose deletion leaves a forest with two components. We may therefore assume $|V(G)| = 6$. Let $y$ denote the remaining vertex.

    Consider the cut-set $K_{vz}$ consisting of edges incident to $vz$. Deleting $K_{vz}$ leaves a forest with two components unless both or neither of $yu$ and $yx$ are edges. By symmetry, we are done unless both $yv$ and $yz$ are edges. Since $y$ must be connected to something, we may assume after relabeling that $yu$ and $yx$ are edges.

    Now similarly consider the cut-sets $K_{ux}$ and $K_{uy}$ of edges incident to $ux$ and $uy$, respectively. Arguing as in the previous paragraph and using the fact that $u$ is not a cut-vertex, we may assume after relabeling that $yv$ and $yz$ are edges.

    Finally consider the cut-sets $K_{uy}$ and $K_{vy}$ defined similarly. The first is a solution to our problem unless both or neither of $vx$ and $xz$ are edges, and the second is unless both or neither of $uz$ and $xz$ are edges. Thus, $G$ has either none or all three of the edges $uz$, $vx$, and $xz$. Correspondingly, $G$ equals $G_3$ or $G_4$.
\end{proof}

The following example shows that Theorem \ref{thm:rank-5 cocircuits} does not hold in rank four and lower. 

\begin{example}[Rank-four counterexample] 
Let $M$ be $M(C_5)$ and let $\mu(e)=1$ for all $e\in E(M)$. If $N=M$, then for any two disjoint cocircuits $C_1^*$ and $C_2^*$ of $N$, we have $\mu(C_1^*)=2$ and $\mu(N\backslash C_2^*)=3$, and Condition (i) of Corollary~\ref{cor:MinimalModularColoops} fails. Therefore, $N$ must be a proper induced restriction of $M$. Since the result requires $r(N)\geq 3$, the only choice for $N$ up to isomorphism is $U_{3,3}$. In this case, $\mu(C^*)=1$ and $\mu(N\backslash C^*)=2$ for each cocircuit $C^*$ of $N$,
%. Moreover, for any two disjoint cocircuits $C_1^*$ and $C_2^*$ of $N$, we have $\mu(C_1^*)=1$ and $\mu(N\backslash C_2^*)=2$, 
and again Condition (i) of Corollary~\ref{cor:MinimalModularColoops} fails.
\end{example}

Finally, we note that rank four is home to a version of Corollary \ref{cor:MinimalModularColoops} and Theorem~\ref{thm:rank-5 cocircuits} that has a weaker assumption (in the sense of a smaller matroid to work with) and correspondingly a weaker conclusion. It is due to Nienhaus and is a key tool in his paper \cite{Nienhaus-pre}. For clarity and to demonstrate the extent to which Nienhaus' results motivated those in this paper, we state all three of his structure results in the following:

\begin{theorem}[Nienhaus]
Let $M$ be a simple regular matroid in rank four.
    \begin{enumerate}
        \item (Double coloop splitting) Either $M$ has a cocircuit $C^*$ such that $M \backslash C^* = U_{3,3}$, or $M$ is a graphic matroid on $K_5$ or $K_3 \oplus K_3$.
        \item (Locally minimal coloop splitting) For any $\mu \colon E(M) \to \mathbb R$, there exists cocircuit $C^*$ such that $M \backslash C^* = U_{1,1} \oplus N$ and $\mu(U_{1,1}) \leq \mu(f)$ for all $f \in E(N)$.
        \item For any $\mu: E(M) \to \mathbb R$, there exists an induced restriction $N$ of rank three that contains disjoint cocircuits, $C_1^*$ and $C_2^*$, with $2 \mu(C_1^*) \leq \mu(N \backslash C_2^*)$.
    \end{enumerate}
\end{theorem}

\section{Application to Geometry}\label{sec:Geometry}

In this section, we translate the matroid-theoretic result, Corollary \ref{cor:MinimalModularColoops}, into the representation-theoretic result, Theorem \ref{thm:Representation}, we need in the geometric applications. We then prove Theorem \ref{thm:Ric_2} and analogous results in odd dimensions (see Theorem \ref{thm:Ric_2-odd} and Corollary \ref{cor:Ric_2-odd}). 

\begin{theorem}
\label{thm:Representation}
    If $\gT^6 \to \SO(V)$ is a faithful representation of a torus of rank six, then there exists a subgroup $\gH \subseteq \gT^6$ of dimension two and subgroups $\gH_i \subseteq \gH$ of dimension three for $i \in \{1, 2\}$ whose fixed-point sets $V^{\gH}$ and $V^{\gH_i}$ satisfy all of the following:
        \begin{enumerate}
            \item $V^{\gH_1}$ and $V^{\gH^2}$ are proper subspaces of $V^{\gH}$ that intersect transversely, 
            \item $3 \cod\of{V^{\gH_1} \subseteq V^{\gH}} \leq \dim V^{\gH_2}$, and
            \item $\dim\of{V^{\gH_1} \cap V^{\gH_2}} \equiv \dim V^{\gT^6} \bmod 4$.
        \end{enumerate}
\end{theorem}

\begin{proof}
    We may assume that the representation has connected isotropy groups, since otherwise we can pass to the fixed-point set of a maximal finite isotropy group and conclude the result by induction over the dimension of $V$ (see \cite[Lemma 1.11]{KWW1}.

    Similarly, we may assume there are no trivial representations since we could replace $V$ by the orthogonal complement $V^{\gT^6}$ and again apply induction.

    Since $\rho$ has connected isotropy groups, its weights may be viewed as a regular matroid (see \cite{KWW2}). The condition on no trivial representations implies that this matroid has no loops. We also simplify the matroid by removing parallel edges and keep track of multiplicities by attaching integer weights to the ground set of this matroid.
    
    Corollary \ref{cor:MinimalModularColoops} now implies the result, as we explain. We recall that cocircuit deletions in the matroid correspond to passing to fixed-point sets of one-dimensional subgroups with the property that the induced action on the fixed-point set are faithful (see \cite[Section 1.4]{KWW2}). The induced restriction of rank four in Corollary~\ref{cor:MinimalModularColoops} corresponds to the fixed-point set of the codimension-four subgroup $\gH \subseteq \gT^6$, and the disjoint cocircuits satisfying the two conditions in that corollary correspond to the subgroups $\gH_1$ and $\gH_2$ whose codimension in $\gH$ is one. The disjointness implies (1), and the two conditions in the corollary imply (2) and (3). 
\end{proof}

We use this to improve a result implicit in the \cite[Proof of Theorem A]{KennardMouilleNienhaus}. Namely, that proof uses the following lemma but with our $\gT^6$ replaced by $\gT^7$.

\begin{lemma}
\label{lem:t6}
    If $F^{4f}$ is a fixed-point component of an effective $\gT^6$-action by isometries on a closed, orientable Riemannian manifold $M$ with $\Ric_2 > 0$, then $F$ has vanishing odd Betti numbers.
\end{lemma}

\begin{proof}
    Choose subgroups $\gH_i \subseteq \gH \subseteq \gT^6$ for $i \in \{1, 2\}$ as in Theorem \ref{thm:Representation}, and let $N$ be the fixed-point component of $\gH$ that contains $F^{4f}$, and let $N_i$ denote the component of $\gH_i$ containing $F^{4f}$ for $i \in \{1,2\}$.

    By the transversality condition together with the codimension condition in Theorem~\ref{thm:Representation}, the inclusion $N_1 \cap N_2 \to N_2$ is $c$-connected with $c = \dim(N_1 \cap N_2) - 1$ by the Connectedness Lemma (see \cite[Remark 2.4]{Wilking03}). By the Periodicity Lemma, the cohomology of $N_2$ is $k$-periodic on degrees $1 \leq * \leq \dim N_2 - 1$ for some $k \leq \tfrac 1 3 \dim N_2$ (see \cite[Lemma 2.2]{Wilking03} and \cite[Definition 1.1 and Example 1.2]{KennardMouilleNienhaus}). By the Partial Four Periodicity Theorem in \cite[Theorem 4.2]{KennardMouilleNienhaus}, the rational cohomology of $N_2$ is four-periodic on degrees $1 \leq * \leq \dim N_2 - 1$. Using the connectedness of the inclusion $N_1 \cap N_2 \to N_2$, we also have that $N_1 \cap N_2$ has four-periodic rational cohomology on degrees $1 \leq * \leq \dim(N_1 \cap N_2) - 1$. Combined with the congruence condition on the dimension of $N_1 \cap N_2$, this manifold has vanishing Betti numbers in odd degrees by Poincar\'e duality. Finally, the same holds for $F$ since it arises as a fixed-point component of a torus action on $N_1 \cap N_2$ (see \cite[Section 1]{KWW1}).
\end{proof}

With this lemma in place, the main theorem follows as in \cite{KennardMouilleNienhaus}. We only provide the sketch to explain the improvements where relevant to the torus bounds.

\begin{proof}[Proof sketch of Theorem \ref{thm:Ric_2}]
We are given an isometric $\gT^9$-action on a closed, orientable Riemannian manifold $M^{4n}$ with $\Ric_2 > 0$. By the Isotropy Rank Lemma, there exists a subgroup $\gT^7$ with a non-empty fixed-point set. Fix a component $F$ of this fixed-point set.

By passing to a kernel of an irreducible subrepresentation of odd multiplicity or by replacing $M$ by the fixed-point component containing $F$ of any circle subgroup of $\gT^7$, we may assume by induction on the dimension of $M$ that $F$ is a $\gT^7$-fixed-point component of an induced action on a $\gT^6$-fixed-point component $G$ with dimension divisible by four. By Lemma \ref{lem:t6}, we conclude that $G$ has vanishing rational cohomology in odd degrees, and we can push down to conclude the same for $F$.

Since the previous argument applies for any fixed-point component $F$ of the $\gT^7$-action on $M$, and since there is at least one such component, we have
    \[\chi(M) = \sum \chi(F) > 0.\]
This proves the positivity of the Euler characteristic. The stronger claim on the odd Betti numbers of fixed-point components of $\gT^7$-actions follows by the proof. 

For the statement in the case of connected isotropy groups, this follows immediately from the proof of Theorem B in \cite{KennardMouilleNienhaus} together with the fact that the positivity of the Euler characteristic implies that the $\gT^9$-action has a fixed point.
\end{proof}

To close this section we prove analogous geometric results in odd dimensions. The first is related to the fact that the Euler characteristic positivity in Theorem \ref{thm:Ric_2} implies that an $\gT^d$-action has a fixed point in dimensions divisible by four when $d \geq 9$. This reflects what occurs for metrics with positive sectional curvature. The analogous statement for odd dimensions is that some $\gT^{d-1}$ has a fixed-point or, equivalently, that the $\gT^d$ has a circle orbit.

\begin{theorem}\label{thm:Ric_2-odd}
A $(4n+1)$-dimensional closed manifold with a $\gT^9$-invariant Riemannian metric with positive second intermediate Ricci curvature has a circle orbit.
\end{theorem}

\begin{proof}
    The proof is similar to the case in even dimensions. We can pass to a $\gT^7$ with a fixed point, and we reduce further to the case where this fixed point is contained in a fixed-point component $G$ of a subtorus $\gT^6$ such that the codimension of $G$ is divisible by four. In particular, $\dim G$ is of the form $4g+1$. Applying Theorem \ref{thm:Representation} and the Partial Four Periodicity Theorem, we find that $G$ has four-periodic rational cohomology on degrees $1 \leq * \leq 4g$. Using Poincar\'e duality and the fact that $\pi_1(G)$ is finite by the Bonnet-Myers theorem, it follows that $G$ is a rational homology sphere. It holds generally now for smooth torus actions on rational spheres that a circle orbit exists. The entire torus acts on this circle orbit, so the torus has a circle orbit as well on this circle or, if not, nearby by the Slice Theorem.
\end{proof}

Similarly, we derive a corollary that extends \cite[Theorem A]{KWW2} in odd dimensions.

\begin{corollary}\label{cor:Ric_2-odd}
    If $M^{4n+1}$ is a closed Riemannian manifold with positive second intermediate Ricci curvature and an effective, isometric $\gT^{10}$ action with only connected isotropy groups, then $M$ is a rational homology sphere.
\end{corollary}

\begin{proof}
    By Theorem \ref{thm:Ric_2-odd}, there is a circle orbit. Fixing any point in this orbit shows that some subgroup $\gT^9$ has a fixed-point. The proof in \cite{KWW2} now shows that $M$ has $4$-periodic rational cohomology on degrees $1 \leq * \leq \dim M - 1$ (see \cite[Section 5.4]{KWW2} for a direct proof). Since again we have $\dim M \equiv 1 \bmod 4$, Poincar\'e duality implies that $M$ is a rational sphere. 
\end{proof}

We remark on the congruence condition in these theorems. 
Additional tools are developed in \cite{KWW1} to extend results into dimensions $4n+2$ and $4n+3$. We do not yet have a $b_3$ Lemma (see \cite[Lemma 3.2]{KWW1}), for example. If one could verify it, then one no longer needs the congruence condition in Theorem~\ref{thm:Representation}. In particular, one could skip the step of passing to a codimension-one subtorus to preserve the required congruence condition. In addition, one could apply Theorem \ref{thm:rank-5 cocircuits}, which only requires a rank-five matroid (and correspondingly a rank-five torus in the application). In other words, one could cut by two the rank of the torus. 
We propose this here as a problem for future study.
\begin{conjecture}

\label{con:T7}
An even-dimensional closed manifold admitting a $T^7$-invariant Riemannian metric with positive $\Ric_2$ has positive Euler characteristic.
\end{conjecture}

Even with this conjectured improvement, the question would remain whether the torus rank could be reduced to six or to five. The most symmetric example known so far with even dimension, positive second intermediate Ricci curvature, and {\it Euler characteristic zero} has torus symmetry of rank four (see \cite{D_DV_GA_RV-pre}).

% \bibliographystyle{alpha}
% \bibliography{references}

\end{document}